\documentclass{article}
\usepackage{subfig}
\usepackage[section] {placeins}
\usepackage{amsthm}
\usepackage{amssymb,amsmath, mathdots}
\usepackage{anysize}
\usepackage{psfrag}
\newtheorem{thm}{Theorem}[section]
\newtheorem{def.}{Definition}[section]

\newtheorem{lem}{Lemma}[section]

\numberwithin{table}{section}
\begin{document}

\title{Delta Diagrams}
\author{Slavik Jablan\\
        The Mathematical Institute\\
        Knez Mihailova 36\\
        P.O. Box 367, 11001, Belgrade\\
        Serbia\\
        \texttt{sjablan@gmail.com}\\
        and\\
        Louis H. Kauffman\\
        Department of Mathematics, Statistics and Computer Science\\
        University of Illinois at Chicago\\
        851 S. Morgan St., Chicago IL 60607-7045\\
        USA\\
        \texttt{kauffman@uic.edu}\\
        and\\
        Pedro Lopes\\
        Center for Mathematical Analysis, Geometry and Dynamical Systems\\
        Department of Mathematics\\
        Instituto Superior T\'ecnico, Universidade de Lisboa\\
        1049-001 Lisboa\\
        Portugal\\
        \texttt{pelopes@math.tecnico.ulisboa.pt}\\
}
\date{December 31, 2015}
\maketitle

\begin{abstract}
We call a {\it Delta Diagram} any diagram of a knot or link whose regions (including the unbounded one) have  3, 4, or 5 sides. We prove that any knot or link admits a delta diagram. We define and estimate combinatorial link invariants stemming from this definition.
\end{abstract}

\bigbreak

Keywords: knots, links, diagrams, Delta diagrams, lune-free diagrams.

\bigbreak

MSC 2010: 57M27

\bigbreak

\section{Introduction}

\noindent

This article is about the graphical structure of a link diagram. Specifically, we prove that any link can be represented by a diagram whose regions (including the unbounded one) possess $3$, $4$, or $5$ sides; we call these {\it Delta Diagrams}. More precisely, we prove that, starting from a braid closure representation of the link, we can deform it in order to obtain an equivalent delta diagram, see Theorem \ref{thm:deltadiagram}.

Delta diagrams are particular cases of lune-free diagrams -- diagrams which do not possess regions with $2$ sides (lunes). In this context, the present article is a sequel to \cite{JKL}, where we proved that any link diagram can be deformed into a lune-free diagram. We called this process ``delunification''. We showed this process can be realized via finite sequences of Reidemeister moves. Furthermore, we proved we can choose these finite sequences of Reidemeister moves so that a diagram equipped with a non-trivial Fox coloring, preserves its number of colors all the way through its delunification. Thus, the minimum number of colors is attained on lune-free diagrams and we can search for these minima over the smaller class of lune-free diagrams. As of $16$ crossings things become intractable. Perhaps with a subclass of lune-free diagrams things would become intractable at a (much) later stage? This prompted us into the investigations which led to the current article.

We warn the reader that the links studied in this article are non-split links. That is, they do not isotope into disconnected components. The reason is that disconnected components can be resolved into delta diagrams on their own.

The article is organized as follows. In Section \ref{sect:prelim} we introduce the relevant definitions and the auxiliary results. In Section \ref{sect:proof} we prove the main result, the fact that a knot or link can be represented by a diagram whose regions have $3$, $4$, or $5$ sides -- called below delta diagrams.  In Subsection \ref{subsect:conseqkinks} we present an upper bound on the number of crossings produced by the transformation into a delta diagram. In Section \ref{sect:discharging} we prove that a connected lune-free diagram has at least one triangle which is adjacent to another triangle or to a tetragon or to a pentagon (Theorem \ref{thm:dis}). We prove this result using the method of discharging. This result prompts us to define invariants which we list in Subsection \ref{subsect:featureinvs}. In Section \ref{sect:invs} we define further invariants associated to the delta diagrams and lune-free diagrams. Finally, in Section \ref{sect:questions} we list a number of questions for future research.

The article \cite{adams} came to our attention after writing the present one. In it the authors give a different proof that links have delta diagrams. We feel our proof is significantly different from theirs to justify publication of the current article.

\subsection{Acknowledgments}

P.L. acknowledges support from FCT (Funda\c c\~ao para a Ci\^encia e a Tecnologia), Portugal, through project FCT EXCL/MAT-GEO/0222/2012, ``Geometry and Mathematical Physics''. We acknowledge the  contributions of the referee -- in particular, Subsection \ref{subsect:conseqkinks} stems from the referee's remarks.

\section{Preliminaries}\label{sect:prelim}

\begin{def.}
We call {\rm delta diagram} any diagram of a knot or link whose regions have  3, 4, or 5 sides.
\end{def.}

In Figure \ref{fig:hopflink} we illustrate how to obtain a delta diagram for the Hopf link from the braid closure of $\sigma_1^2$.
\begin{figure}[!ht]
	\psfrag{3}{\huge$3$}
	\psfrag{4}{\huge$4$}
	\centerline{\scalebox{.4}{\includegraphics{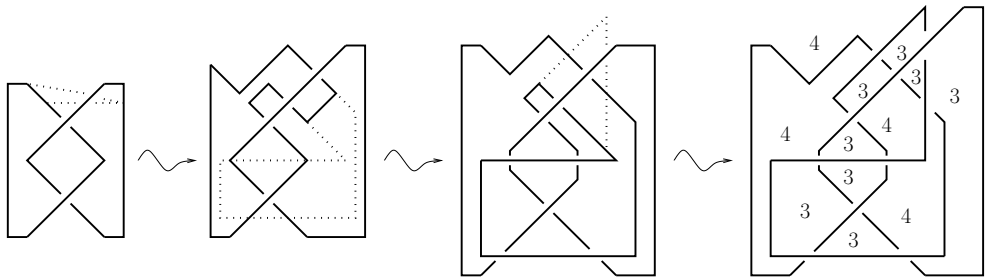}}}
	\caption{Deltization of the Hopf Link. The broken lines indicate the move that takes to the diagram on the right. The numbers on the rightmost diagram count the sides of the regions they are in.}\label{fig:hopflink}
\end{figure}
This treatment of the Hopf link is completely different from that we will be applying to the other knots and links. Our technique will rely on drawing arcs through regions in order to break them down into regions with less sides. At this point the following Lemma is crucial.

\begin{lem}\label{lem:main}
In order to break down a region into regions with less sides by drawing an arc through it, the arc has to skip two sides. Regions with $3$, $4$ or $5$ sides cannot be broken down into regions with less sides $($greater than $2)$.
\end{lem}
\begin{proof}
\begin{figure}[!ht]
	\psfrag{3}{\huge$3$}
	\psfrag{4}{\huge$4$}
	\psfrag{1}{\huge$1$}
	\psfrag{5}{\huge$5$}
	\psfrag{n}{\huge$n$}
	\psfrag{n+1}{\huge$n+1$}
	\psfrag{n-1}{\huge$n-1$}
	\psfrag{n-2}{\huge$n-2$}
	\psfrag{n-3}{\huge$n-3$}
	\centerline{\scalebox{.27}{\includegraphics{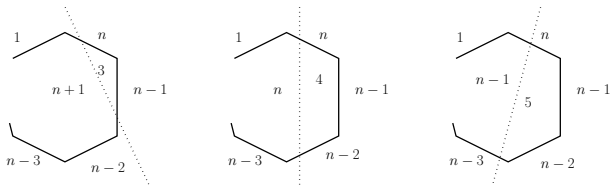}}}
	\caption{An $n$-side region and an arc driven through it skipping $0$ sides (leftmost case), $1$ side (middle case), or $2$ sides (rightmost case). The numbers outside the $n$-sided region enumerate the sides, the numbers inside the regions count the number of sides.}\label{fig:skip2lemma}
\end{figure}
\begin{figure}[!ht]
	\psfrag{3}{\huge$3$}
	\psfrag{4}{\huge$4$}
	\psfrag{1}{\huge$1$}
	\psfrag{5}{\huge$5$}
	\psfrag{n}{\huge$n$}
	\psfrag{n+1}{\huge$n+1$}
	\psfrag{n-1}{\huge$n-1$}
	\psfrag{n-2}{\huge$n-2$}
	\psfrag{n-3}{\huge$n-3$}
	\centerline{\scalebox{.27}{\includegraphics{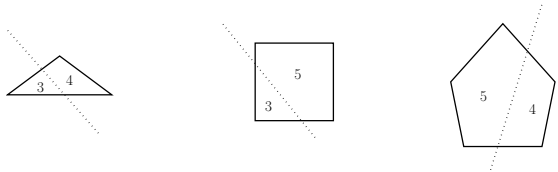}}}
	\caption{Triangles, squares, and pentagons: the net effect of skipping $2$ sides does not result in the breaking down of these polygons into regions with less sides.}\label{fig:skip2lemma345}
\end{figure}
Figure \ref{fig:skip2lemma} shows that it is only as of skipping $2$ sides that an arc drawn through a region is able to break it down into regions with less sides (but more than $2$). Figure \ref{fig:skip2lemma345} shows where this procedure does not work. Specifically, with a $3$-sided region it is not possible to skip $2$ sides. With a $4$-sided region if $2$ sides are skipped, then $0$ sides will be also skipped. With a $5$-sided region if $2$ sides are skipped than $1$ side will also be skipped. This concludes the proof.
\end{proof}

Leaning on the representation of knots or links as closures of braids, for braid words with more than two letters, we prove that any knot or link has a delta diagram. We develop a  technique based on the introduction of kinks (see Figure \ref{fig:kinkification}) that allows us to obtain a delta diagram equivalent to  the braid closure form in which the knot under study was originally represented. We remark that final product is no longer a braid closure.

Figures \ref{fig:10-35} through \ref{fig:10-35-5KINK} below illustrate the case for knot $10_{35}$ regarded as the braid closure of
\[
\sigma_1^{-1}\sigma_2\sigma_3^{-1}\sigma_2\sigma_3^{-1}\sigma_4\sigma_5^{-1}\sigma_4\sigma_5^{-1}\sigma_1\sigma_2^{-1}\sigma_3\sigma_4^{-1}\sigma_2\sigma_2
\]
(\cite{VJones}). In particular the reader should examine first Figure \ref{fig:10-35} and then Figure \ref{fig:10-35-5KINK} since they stand for the beginning and the end of the introduction of kinks. The reader should note that in Figure \ref{fig:10-35} there are regions which do not have the appropriate number of sides, namely, with 2 or with 6 sides in this example (Figure \ref{fig:10-35}); by Figure \ref{fig:10-35-5KINK} the diagram has been transformed into a delta diagram.

Specifically, we introduce a sequence of nested kinks from arc $1$  around the braid entanglement, in the first part of the process, to break down the majority of regions whose number of sides is not $3$, nor $4$, nor $5$, see Figures \ref{fig:10-35-1KINK} through \ref{fig:10-35-4KINK}. Then we introduce a kink from arc $2$ to break down any region left over from Step $1$ (or created in Step $1$), which does not have the appropriate number of sides into regions with the appropriate number of sides, see Figure \ref{fig:10-35-5KINK}.

\begin{figure}[!ht]
	\psfrag{a1}{\Huge\text{arc $1$}}
	\psfrag{k}{\Huge\text{kinkification}}
	\psfrag{s}{\Huge\text{spiralling}}
	\psfrag{...}{$\dots$}
	\centerline{\scalebox{.2}{\includegraphics{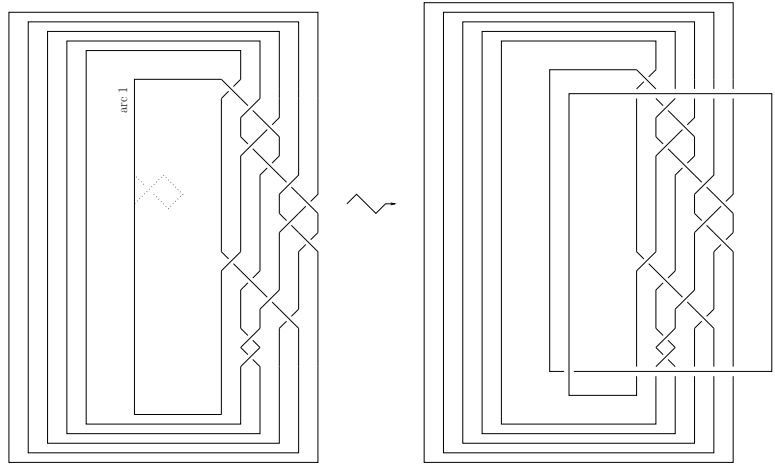}}}
	\caption{Illustrating the  kinkification.}\label{fig:kinkification}
\end{figure}

\begin{figure}[!ht]
	\psfrag{2}{\huge$2$}
	\psfrag{6}{\huge$6$}
	\centerline{\scalebox{.3}{\includegraphics{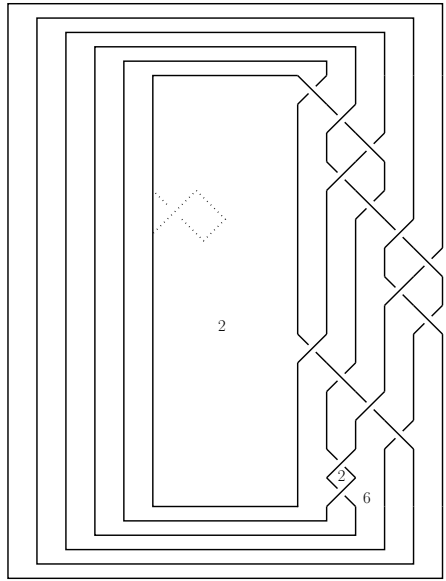}}}
	\caption{Knot $10_{35}$ as the braid closure of $\sigma_1^{-1}\sigma_2\sigma_3^{-1}\sigma_2\sigma_3^{-1}\sigma_4\sigma_5^{-1}\sigma_4\sigma_5^{-1}\sigma_1\sigma_2^{-1}\sigma_3\sigma_4^{-1}\sigma_2\sigma_2$.}\label{fig:10-35}
\end{figure}

\begin{figure}[!ht]
	\psfrag{a1}{\Huge\text{arc $1$}}
	\centerline{\scalebox{.2}{\includegraphics{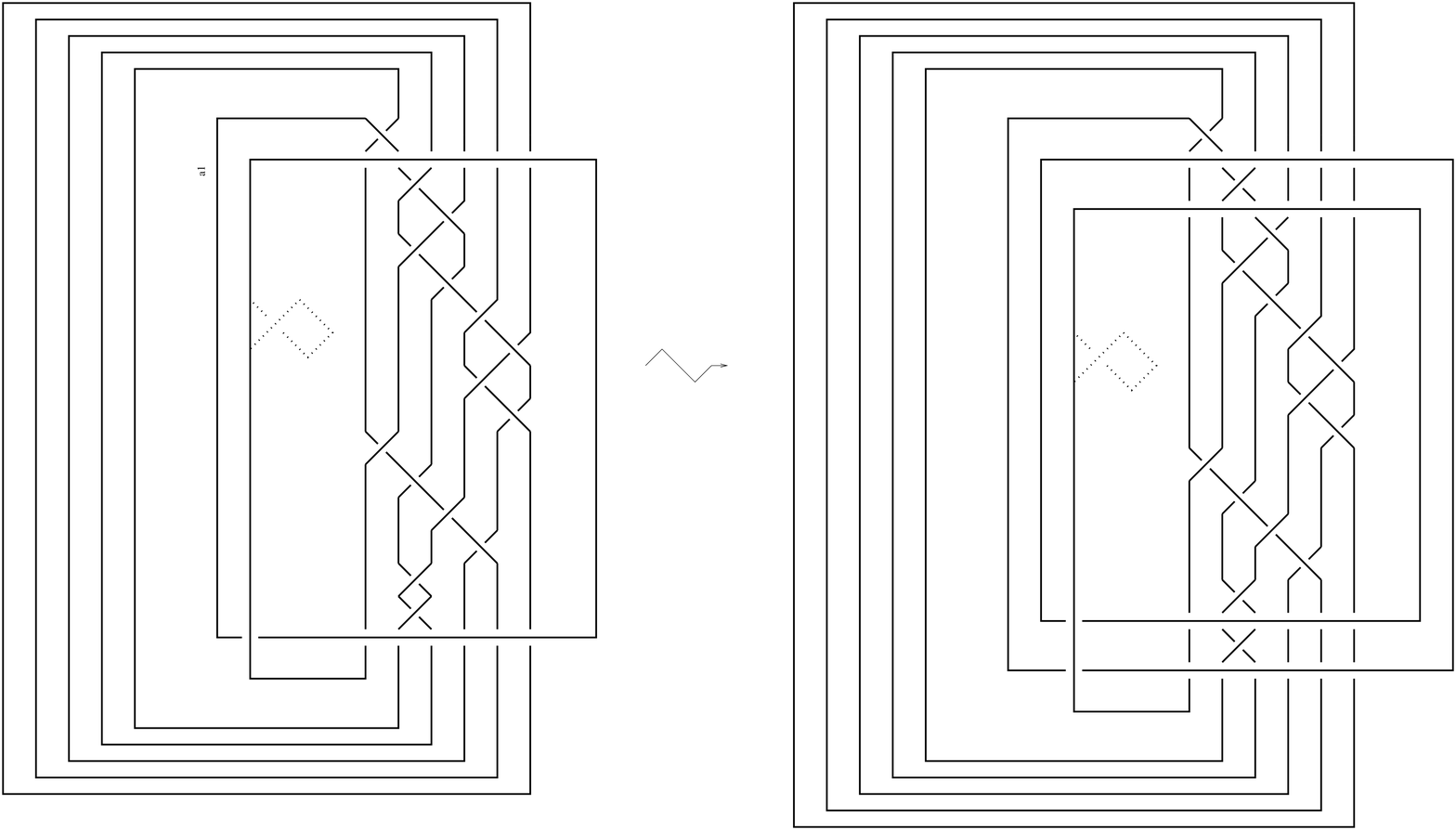}}}
	\caption{Step $1$: introduction of kinks from arc $1$ in order to break down regions within the braid entanglement which do not have the appropriate number of sides.}\label{fig:10-35-1KINK}
\end{figure}
\begin{figure}[!ht]
	\psfrag{a1}{$arc 1$}
	\psfrag{...}{$\dots$}
	\psfrag{max}{\huge\text{(Non-Isolated Lune) Maximal Tassel}}
	\psfrag{nonmax}{\huge\text{(Non-maximal) Tassel}}
	\psfrag{2b-a1}{\Large$2b-a$}
	\psfrag{3b-2a}{\huge$3b-2a$}
	\psfrag{3b-2a1}{\Large$3b-2a$}
	\psfrag{4b-3a}{\large$4b-3a$}
	\centerline{\scalebox{.2}{\includegraphics{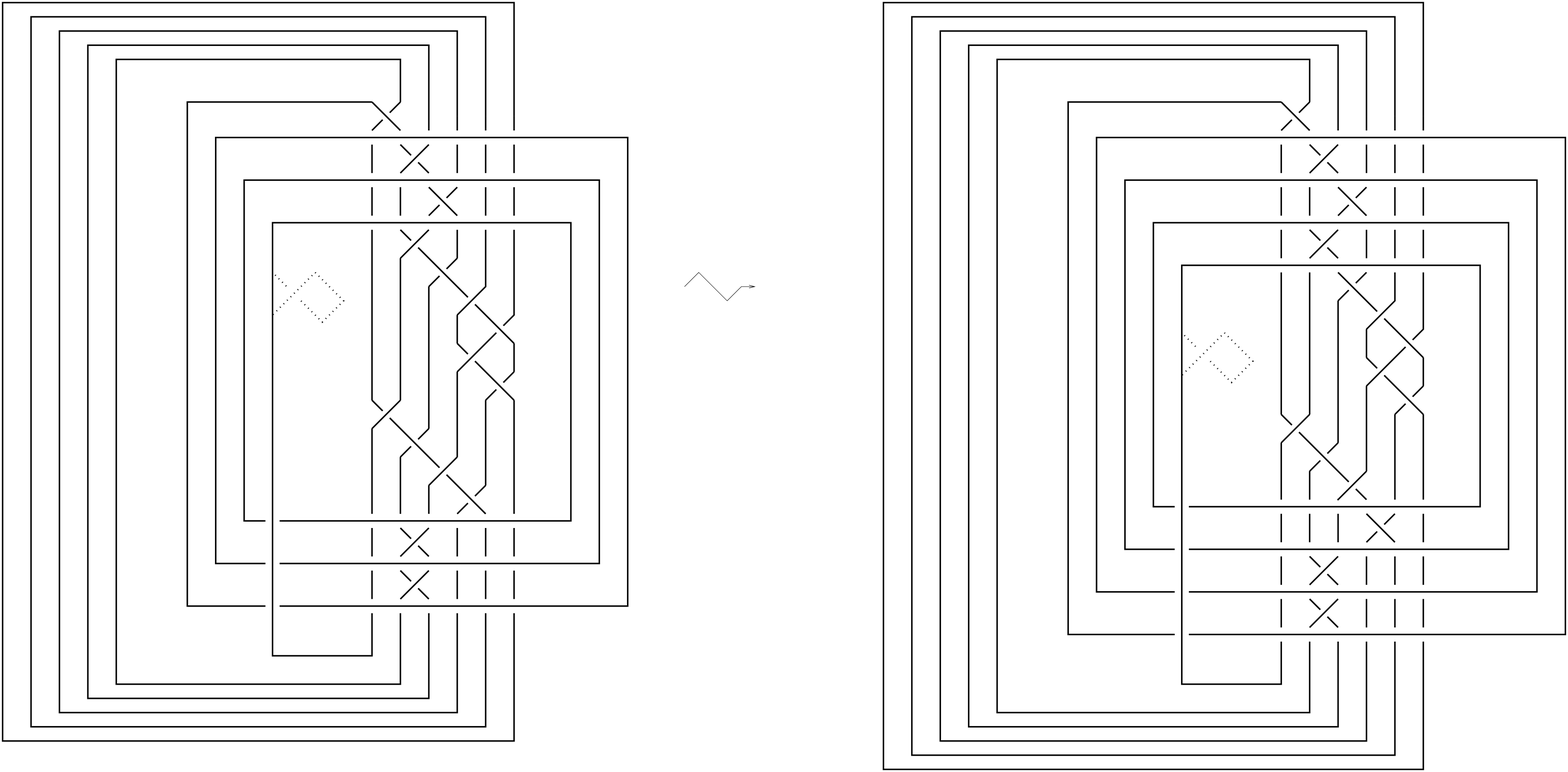}}}
	\caption{Step $1$: further introduction of kinks from arc $1$ in order to break down regions within the braid entanglement which do not have the appropriate number of sides (cont'd).}\label{fig:10-35-2KINK}
\end{figure}

\begin{figure}[!ht]
	\psfrag{a1}{$arc 1$}
	\psfrag{...}{$\dots$}
	\psfrag{max}{\huge\text{(Non-Isolated Lune) Maximal Tassel}}
	\psfrag{nonmax}{\huge\text{(Non-maximal) Tassel}}
	\psfrag{2b-a1}{\Large$2b-a$}
	\psfrag{3b-2a}{\huge$3b-2a$}
	\psfrag{3b-2a1}{\Large$3b-2a$}
	\psfrag{4b-3a}{\large$4b-3a$}
	\centerline{\scalebox{.2}{\includegraphics{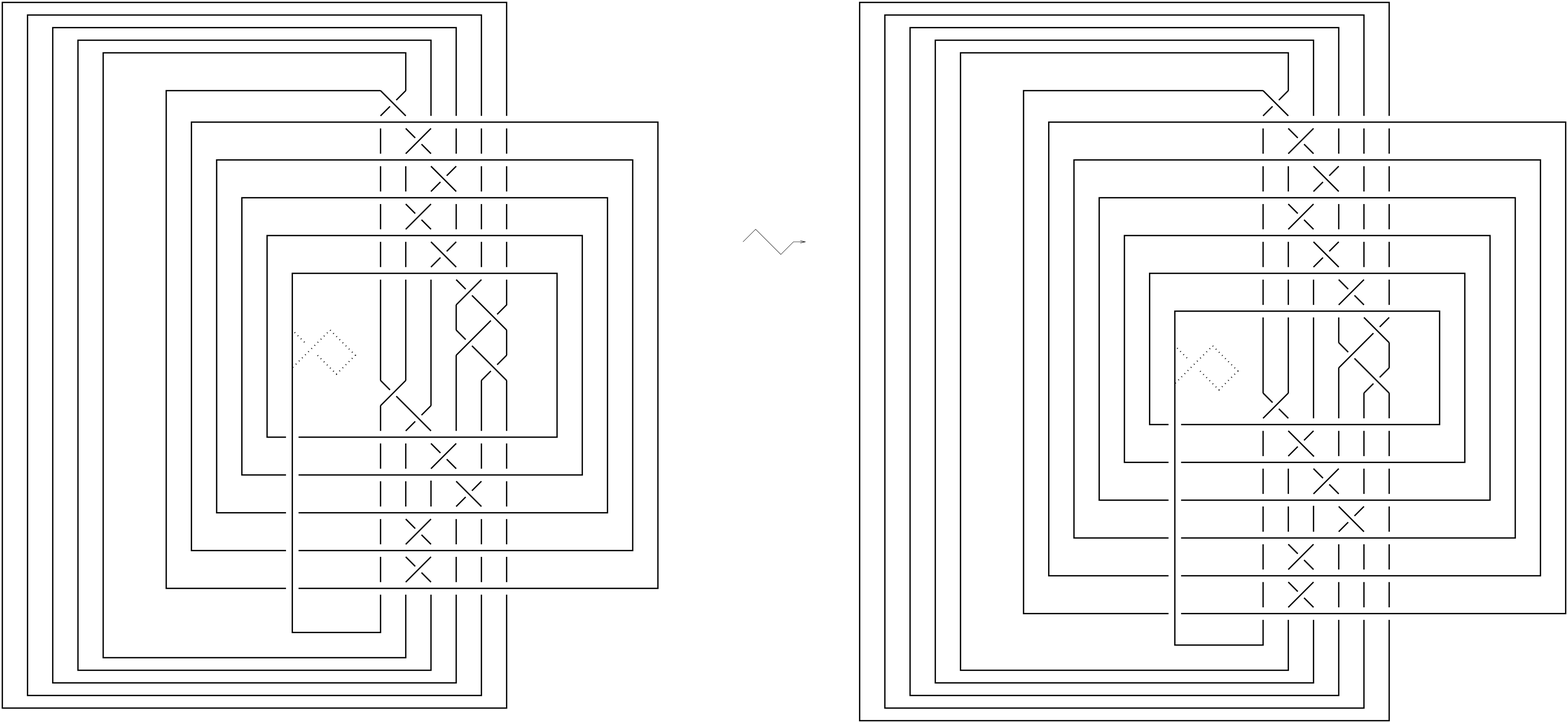}}}
	\caption{Step $1$: further introduction of kinks from arc $1$ in order to break down regions within the braid entanglement which do not have the appropriate number of sides (cont'd).}\label{fig:10-35-3KINK}
\end{figure}

\begin{figure}[!ht]
	\psfrag{a1}{$arc 1$}
	\psfrag{X}{\huge$X$}
	\psfrag{max}{\huge\text{(Non-Isolated Lune) Maximal Tassel}}
	\psfrag{nonmax}{\huge\text{(Non-maximal) Tassel}}
	\psfrag{2b-a1}{\Large$2b-a$}
	\psfrag{3b-2a}{\huge$3b-2a$}
	\psfrag{3b-2a1}{\Large$3b-2a$}
	\psfrag{4b-3a}{\large$4b-3a$}
	\centerline{\scalebox{.2}{\includegraphics{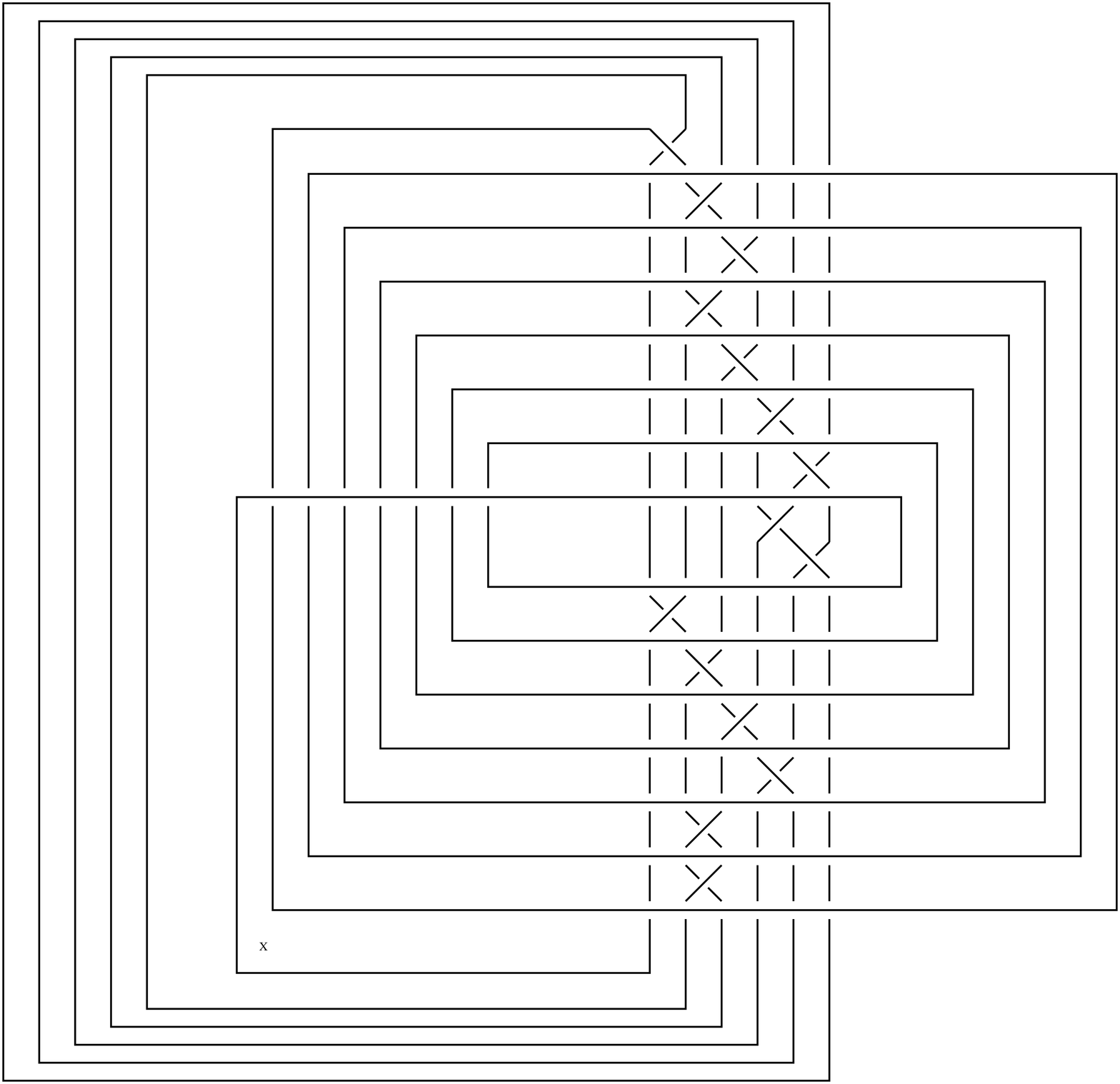}}}
	\caption{Step $1$: further introduction of kinks from arc $1$ in order to break down regions within the braid entanglement which do not have the appropriate number of sides (concl'd). The end of Step $1$.}\label{fig:10-35-4KINK}
\end{figure}

\begin{figure}[!ht]
	\psfrag{...}{$\dots$}
	\psfrag{Tms}{\huge\text{THE Middle Section}}
	\psfrag{nms}{\huge\text{a Non-Middle Section}}
	\psfrag{ttbr}{\huge\text{a Top-to-Bottom Region}}
	\psfrag{lr}{\huge\text{two Lateral Regions}}
	\psfrag{2b-a1}{\Large$2b-a$}
	\psfrag{X}{\huge$X$}
	\psfrag{Y}{\huge$Y$}
	\psfrag{3b-2a1}{\Large$3b-2a$}
	\psfrag{4b-3a}{\large$4b-3a$}
	\centerline{\scalebox{.4}{\includegraphics{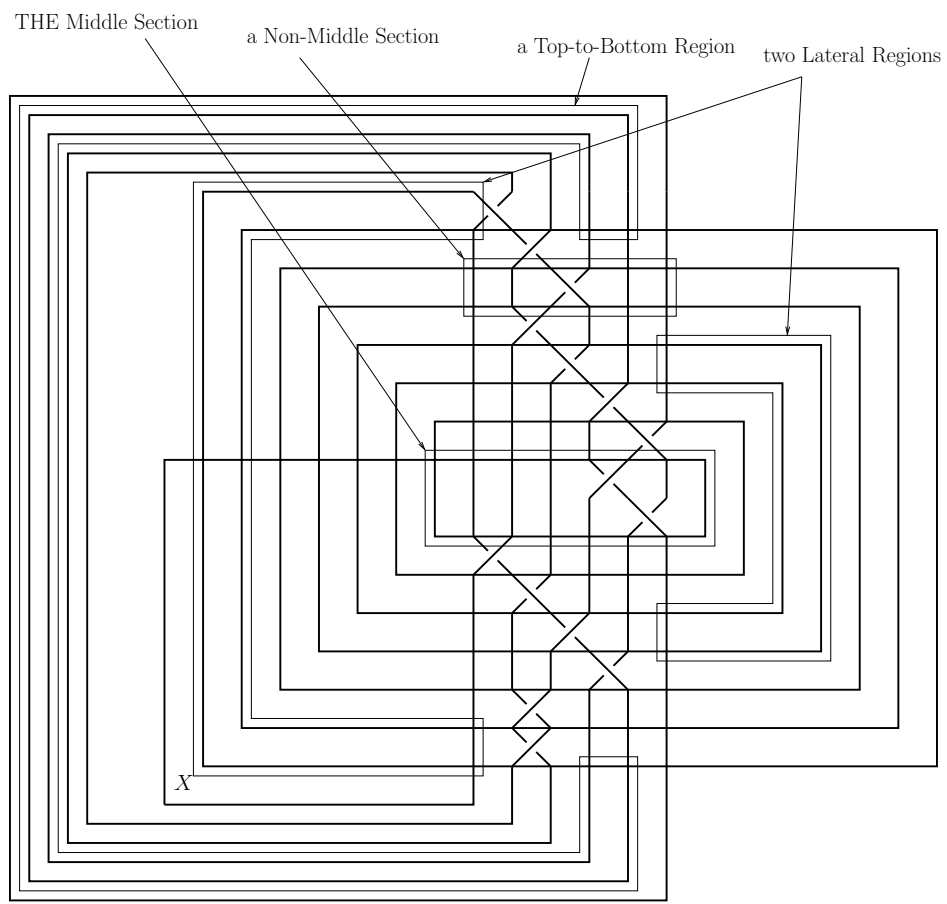}}}
	\caption{Terminology we will be using in the sequel when referring to the status of the process by the end of Step $1$.}\label{fig:terminology}
\end{figure}

\begin{figure}[!ht]
	\psfrag{a2}{\Huge\text{arc $2$}}
	\psfrag{...}{$\dots$}
	\psfrag{max}{\huge\text{(Non-Isolated Lune) Maximal Tassel}}
	\psfrag{nonmax}{\huge\text{(Non-maximal) Tassel}}
	\psfrag{3}{\Large$3$}
	\psfrag{4}{\Large$4$}
	\psfrag{5}{\Large$5$}
	\psfrag{3-}{\large$3$}
	\psfrag{2b-a1}{\Large$2b-a$}
	\psfrag{3b-2a}{\huge$3b-2a$}
	\psfrag{3b-2a1}{\Large$3b-2a$}
	\psfrag{4b-3a}{\large$4b-3a$}
	\centerline{\scalebox{.3}{\includegraphics{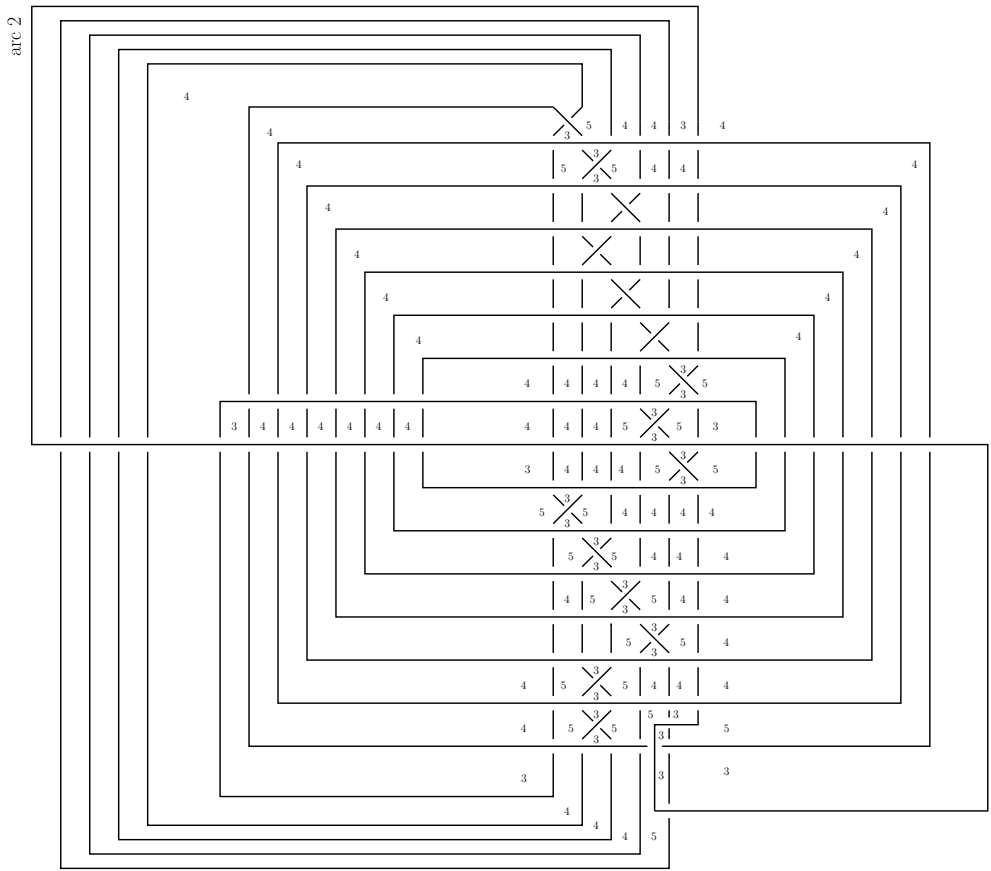}}}
	\caption{Step $2$: introduction of kink from arc $2$ and concluding the saga. The numbers on the regions indicate the number of sides.}\label{fig:10-35-5KINK}
\end{figure}

\section{Proof that knots and links can be represented by delta diagrams.}\label{sect:proof}

\noindent

We now prove the main theorem of this article.

\begin{thm}\label{thm:deltadiagram}
Each knot or link can be represented by a $\Delta$-diagram.
\end{thm}
\begin{proof}

The proof for the Hopf link is in Figure \ref{fig:hopflink}.

The reader should look at Figures \ref{fig:10-35} and \ref{fig:10-35-5KINK} since they stand for the beginning and the end of the kinkification applied to a knot, in this case $10_{35}$.

The proof relies on the fact that each knot or link can be represented as the closure of a braid (cf. \cite{Birman}). We remark that we insist on having a braid word with an even number of $\sigma_i^{\pm 1}$. Should we not have it, we add a last ``$\sigma$'' which is the identity. Furthermore, as illustrated with the knot $10_{35}$ (see Figures \ref{fig:10-35}, \ref{fig:10-35-1KINK}, \ref{fig:10-35-2KINK}, \ref{fig:10-35-3KINK}, \ref{fig:10-35-4KINK}, \ref{fig:10-35-5KINK}, and also Figure \ref{fig:terminology} for terminology), our technique splits the diagram obtained from the closure of a braid into a new diagram whose regions, or certain unions of regions which we call Sections, fall into the following categories, at the end of the Step $1$: Non-middle Sections, the Middle Section, Lateral Regions, Top-to-Bottom Regions, see Figure \ref{fig:terminology}. This proof analyzes these regions and sections, and checks that those that do not yet have the required number of sides by the end of Step $1$, will obtain them upon the performance of Step $2$. We now refer the reader to Figures \ref{fig:nonmiddlesection}, \ref{fig:middlesection}, \ref{fig:lateralregions-new}, \ref{fig:lateralregionsconc}, \ref{fig:toptobottom-new}, \ref{fig:toptobottomregionsc}, \ref{fig:toptobottomconc-new} where these checks are done. This concludes the proof.
\end{proof}

\begin{figure}[!ht]
	\psfrag{...}{$\dots$}
	\psfrag{.}{$\vdots$}
	\psfrag{(i)}{\huge$(i)$}
	\psfrag{(ii)}{\huge$(ii)$}
	\psfrag{(iii)}{\huge$(iii)$}
	\psfrag{(iv)}{\huge$(iv)$}
	\psfrag{(v)}{\huge$(v)$}
	\psfrag{(vi)}{\huge$(vi)$}
	\psfrag{nonmax}{\huge\text{(Non-maximal) Tassel}}
	\psfrag{2}{\Large$2$}
	\psfrag{3}{\Large$3$}
	\psfrag{4}{\Large$4$}
	\psfrag{5}{\Large$5$}
	\psfrag{6}{\Large$6$}
	\psfrag{2b-a1}{\Large$2b-a$}
	\psfrag{3b-2a}{\huge$3b-2a$}
	\psfrag{3b-2a1}{\Large$3b-2a$}
	\psfrag{4b-3a}{\large$4b-3a$}
	\centerline{\scalebox{.4}{\includegraphics{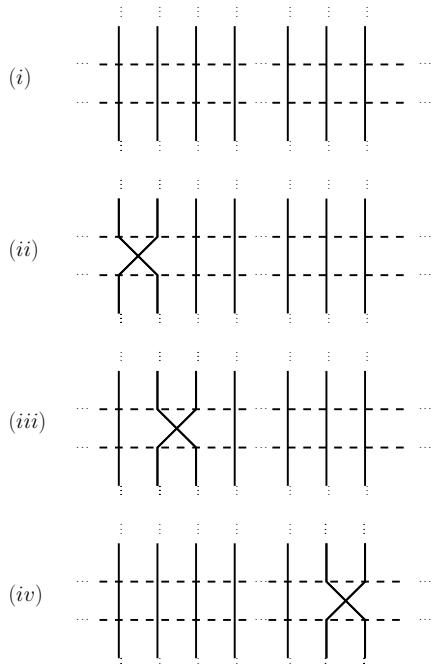}}}
	\caption{Possible instances for non-middle sections. The broken line stands for consecutive stretches of the kinks introduced from arc $1$ around the braid entanglement. Each region in any of these instances has the required number of sides.}\label{fig:nonmiddlesection}
\end{figure}

\begin{figure}[!ht]
	\psfrag{...}{$\dots$}
	\psfrag{.}{$\vdots$}
	\psfrag{(i)}{\huge$(i)$}
	\psfrag{(ii)}{\huge$(ii)$}
	\psfrag{(iii)}{\huge$(iii)$}
	\psfrag{(iv)}{\huge$(iv)$}
	\psfrag{(v)}{\huge$(v)$}
	\psfrag{(vi)}{\huge$(vi)$}
	\psfrag{(vii)}{\huge$(vii)$}
	\psfrag{nonmax}{\huge\text{(Non-maximal) Tassel}}
	\psfrag{2}{\Large$2$}
	\psfrag{3}{\Large$3$}
	\psfrag{4}{\Large$4$}
	\psfrag{5}{\Large$5$}
	\psfrag{6}{\Large$6$}
	\psfrag{2b-a1}{\Large$2b-a$}
	\psfrag{3b-2a}{\huge$3b-2a$}
	\psfrag{3b-2a1}{\Large$3b-2a$}
	\psfrag{4b-3a}{\large$4b-3a$}
	\centerline{\scalebox{.4}{\includegraphics{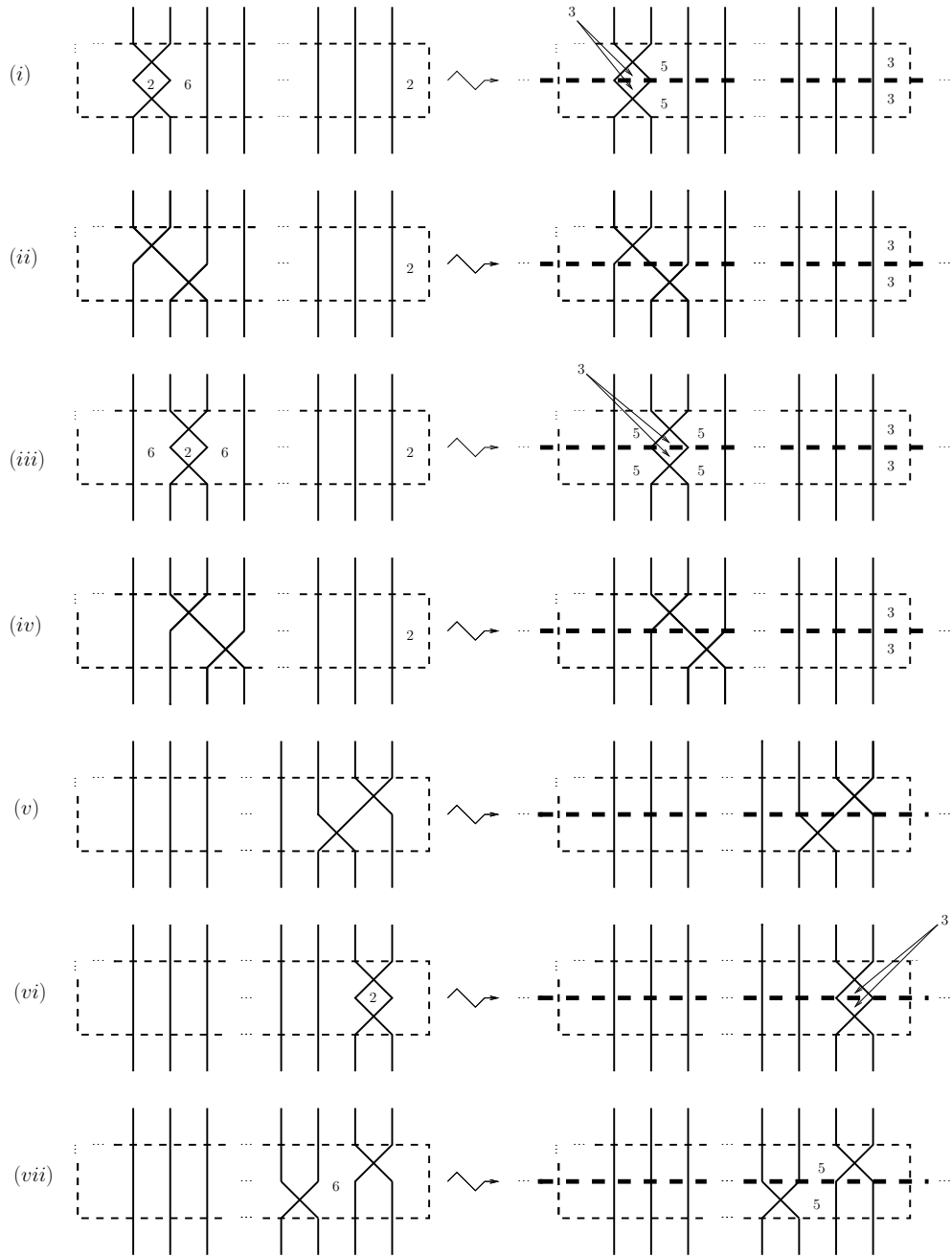}}}
	\caption{Possible instances for the middle section. Left-hand side: last kink introduced from arc $1$ (thin broken line) around the braid entanglement. Right-hand side: the contribution of  arc $2$ (thick broken line) in breaking down the unwanted regions into the right ones. Numbers indicate numbers of sides in relevant regions at this stage of the process. Absence of such numbers means the region at issue already has the expected the number of sides.}\label{fig:middlesection}
\end{figure}

\begin{figure}[!ht]
	\psfrag{a2}{\Large\text{arc $2$}}
	\psfrag{MS}{\Large$\text{middle section}$}
	\psfrag{...}{$\dots$}
	\psfrag{.}{$\vdots$}
	\psfrag{,}{$\iddots$}
	\psfrag{;}{$\ddots$}
	\psfrag{(i)}{\huge$(i)$}
	\psfrag{(ii)}{\huge$(ii)$}
	\psfrag{(iii)}{\huge$(iii)$}
	\psfrag{2}{\Large$2$}
	\psfrag{3}{\Large$3$}
	\psfrag{4}{\Large$4$}
	\psfrag{5}{\Large$5$}
	\psfrag{6}{\Large$6$}
	\centerline{\scalebox{.4}{\includegraphics{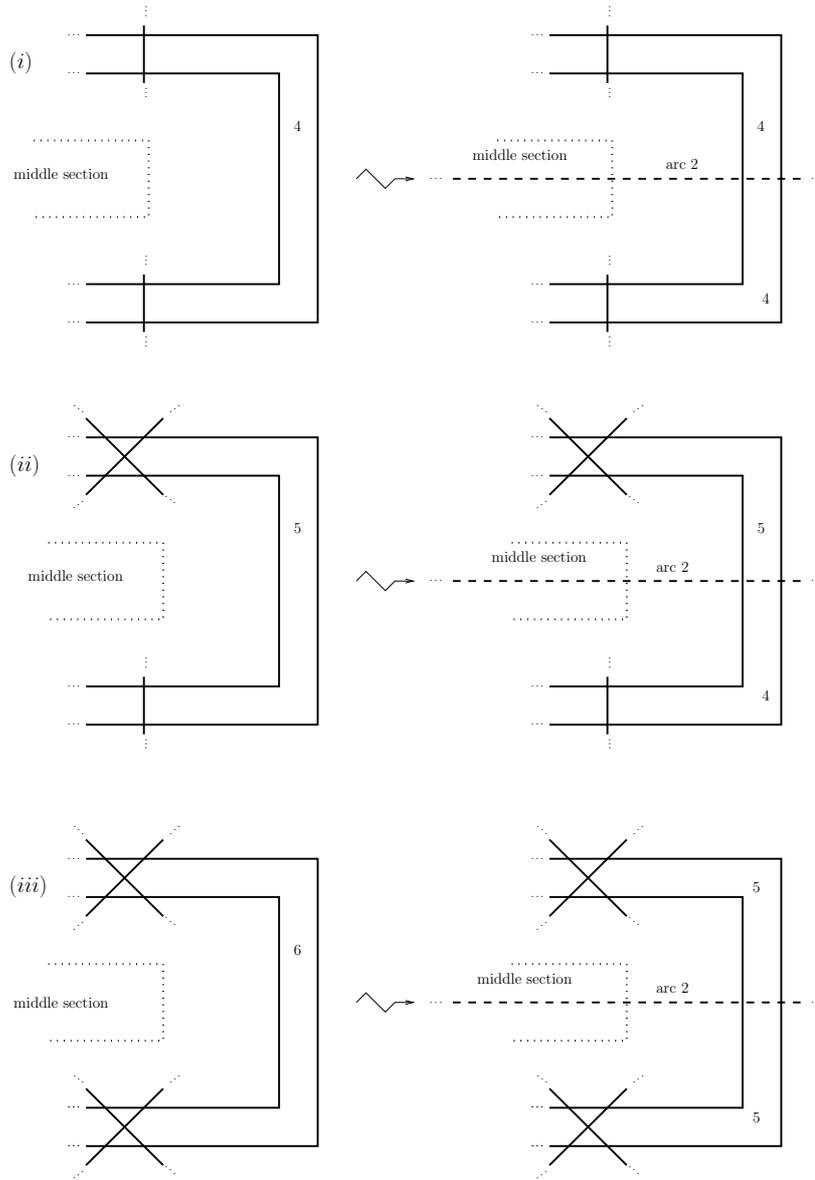}}}
	\caption{Possible instances for the lateral regions. Top row: no changes needed for the lateral regions already have the desired number of sides. Bottom row, left-hand side: last kink introduced from arc $1$ (thin broken line) around the braid entanglement. Bottom row, right-hand side: the contribution of arc $2$ (thick broken line) in breaking down the unwanted regions into the right ones. Numbers indicate numbers of sides in relevant regions at this stage of the process. Absence of such numbers means the region at issue already has the expected number of sides. The lateral regions that show up on the left-hand side of the diagram in Figure \ref{fig:10-35-4KINK} are analogous to the instances $(i)$ through $(iii)$ in the current Figure except that the last swing of the first arc already breaks them down into regions with the required number of sides. There is only one exception here which is the region labeled $X$ in Figure \ref{fig:10-35-4KINK} which is instance $(iv)$ depicted and dealt with in Figure \ref{fig:lateralregionsconc}.}\label{fig:lateralregions-new}
\end{figure}

\begin{figure}[!ht]
	\psfrag{a1}{\Large\text{arc $1$}}
	\psfrag{a2}{\Large\text{arc $2$}}
	\psfrag{MS}{\Large$\text{middle section}$}
	\psfrag{...}{$\dots$}
	\psfrag{.}{$\vdots$}
	\psfrag{,}{$\iddots$}
	\psfrag{;}{$\ddots$}
	\psfrag{(iv)}{\huge$(iv)$}
	\psfrag{2}{\Large$2$}
	\psfrag{3}{\Large$3$}
	\centerline{\scalebox{.4}{\includegraphics{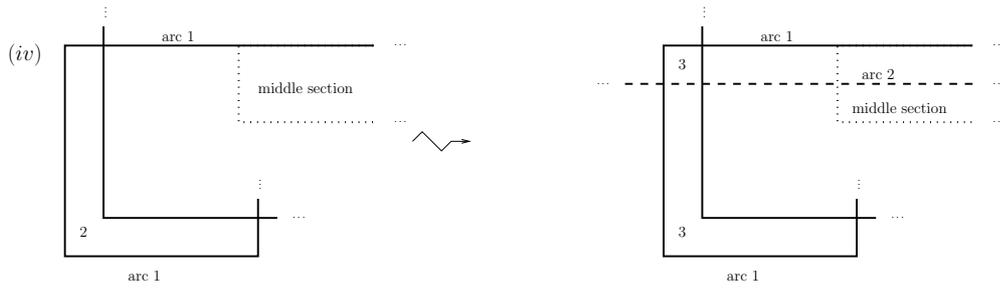}}}
	\caption{Possible instances for the lateral regions (concl'd).}\label{fig:lateralregionsconc}
\end{figure}

\begin{figure}[!ht]
	\psfrag{a2}{\Large\text{arc $2$}}
	\psfrag{MS}{\Large$\text{middle section}$}
	\psfrag{la}{\Large$\text{last arc}$}
	\psfrag{1a}{\Large$\text{first arc}$}
	\psfrag{...}{$\dots$}
	\psfrag{.}{$\vdots$}
	\psfrag{,}{$\iddots$}
	\psfrag{;}{$\ddots$}
	\psfrag{(i)}{\huge$(i)$}
	\psfrag{(ii)}{\huge$(ii)$}
	\psfrag{(iii)}{\huge$(iii)$}
	\psfrag{(iv)}{\huge$(iv)$}
	\psfrag{(v)}{\huge$(v)$}
	\psfrag{(vi)}{\huge$(vi)$}
	\psfrag{nonmax}{\huge\text{(Non-maximal) Tassel}}
	\psfrag{2}{\Large$2$}
	\psfrag{3}{\Large$3$}
	\psfrag{4}{\Large$4$}
	\psfrag{5}{\Large$5$}
	\psfrag{6}{\Large$6$}
	\psfrag{2b-a1}{\Large$2b-a$}
	\psfrag{3b-2a}{\huge$3b-2a$}
	\psfrag{3b-2a1}{\Large$3b-2a$}
	\psfrag{4b-3a}{\large$4b-3a$}
	\centerline{\scalebox{.4}{\includegraphics{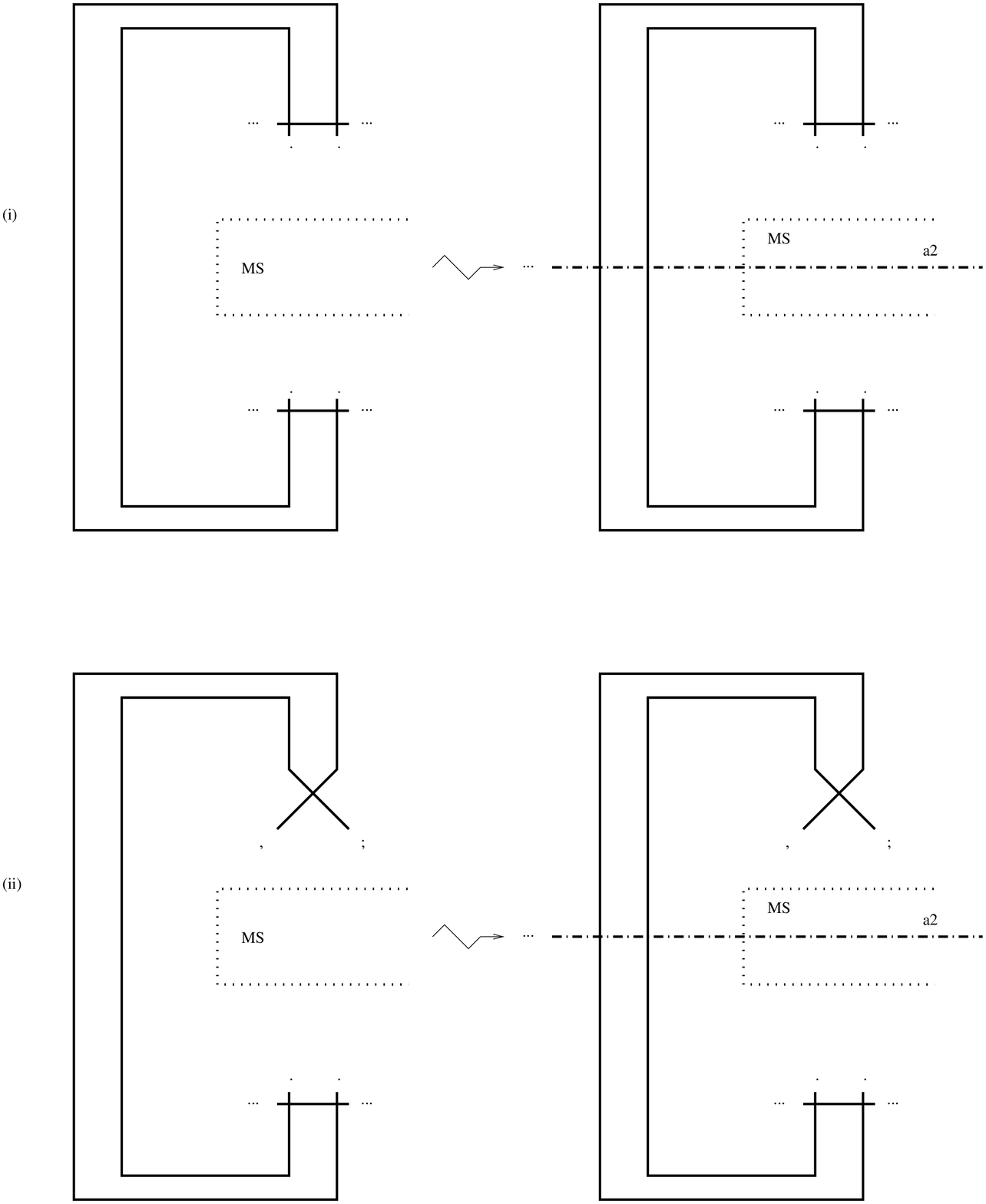}}}
	\caption{Possible instances for the top-to-bottom regions.}\label{fig:toptobottom-new}
\end{figure}

\begin{figure}[!ht]
	\psfrag{a2}{\Large\text{arc $2$}}
	\psfrag{MS}{\Large$\text{middle section}$}
	\psfrag{la}{\Large$\text{last arc}$}
	\psfrag{1a}{\Large$\text{first arc}$}
	\psfrag{...}{$\dots$}
	\psfrag{.}{$\vdots$}
	\psfrag{,}{$\iddots$}
	\psfrag{;}{$\ddots$}
	\psfrag{(i)}{\huge$(i)$}
	\psfrag{(ii)}{\huge$(ii)$}
	\psfrag{(iii)}{\huge$(iii)$}
	\psfrag{(iv)}{\huge$(iv)$}
	\psfrag{(v)}{\huge$(v)$}
	\psfrag{(vi)}{\huge$(vi)$}
	\psfrag{nonmax}{\huge\text{(Non-maximal) Tassel}}
	\psfrag{2}{\Large$2$}
	\psfrag{3}{\Large$3$}
	\psfrag{4}{\Large$4$}
	\psfrag{5}{\Large$5$}
	\psfrag{6}{\Large$6$}
	\psfrag{2b-a1}{\Large$2b-a$}
	\psfrag{3b-2a}{\huge$3b-2a$}
	\psfrag{3b-2a1}{\Large$3b-2a$}
	\psfrag{4b-3a}{\large$4b-3a$}
	\centerline{\scalebox{.4}{\includegraphics{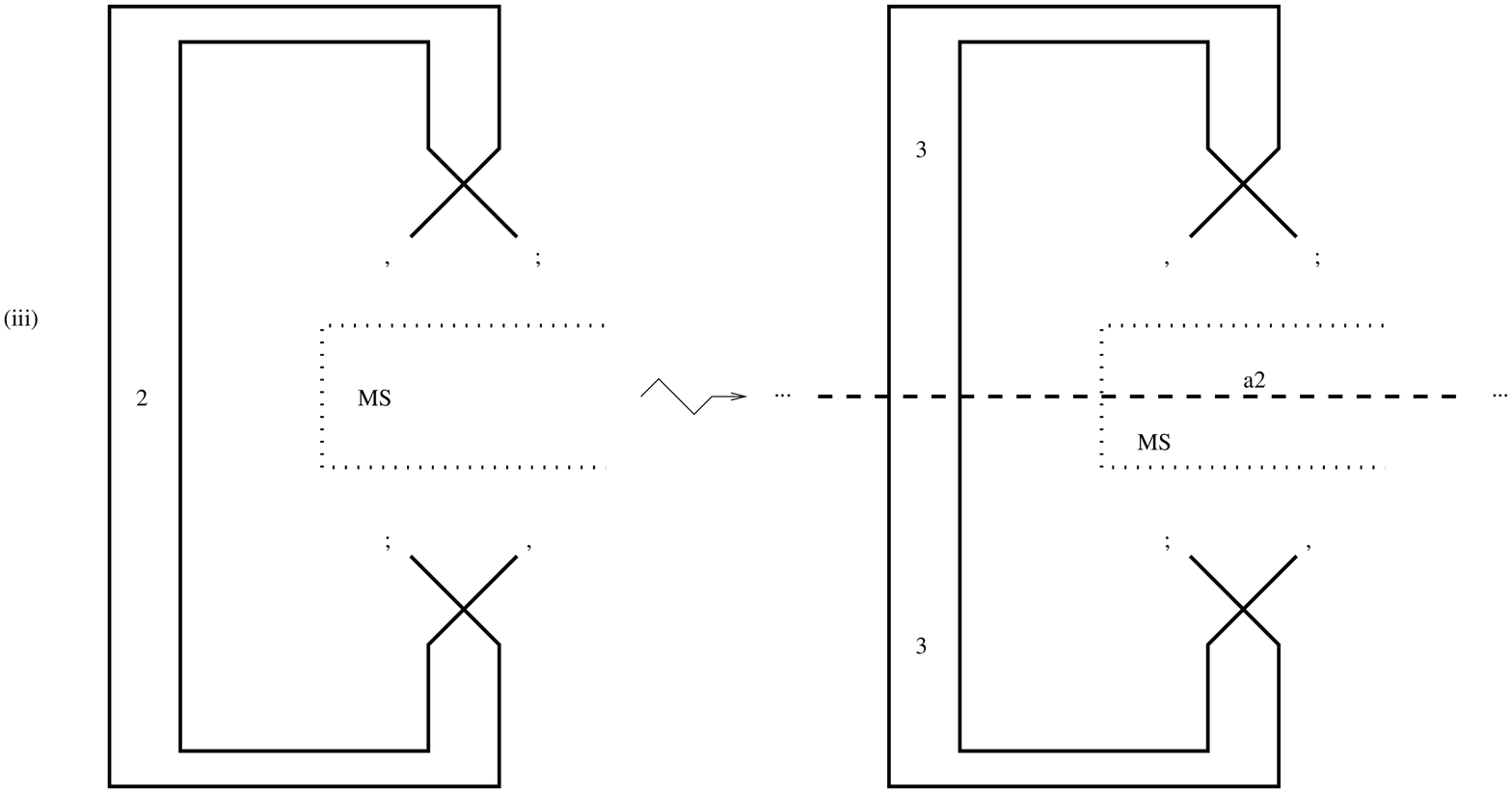}}}
	\caption{Possible instances for the top-to-bottom regions (cont'd).}\label{fig:toptobottomregionsc}
\end{figure}

\begin{figure}[!ht]
	\psfrag{a2}{\Large\text{arc $2$}}
	\psfrag{MS}{\Large$\text{middle section}$}
	\psfrag{la}{\Large$\text{last arc}$}
	\psfrag{1a}{\Large$\text{first arc}$}
	\psfrag{...}{$\dots$}
	\psfrag{.}{$\vdots$}
	\psfrag{,}{$\iddots$}
	\psfrag{;}{$\ddots$}
	\psfrag{(i)}{\huge$(i)$}
	\psfrag{(ii)}{\huge$(ii)$}
	\psfrag{(iii)}{\huge$(iii)$}
	\psfrag{(iv)}{\huge$(iv)$}
	\psfrag{(v)}{\huge$(v)$}
	\psfrag{(vi)}{\huge$(vi)$}
	\psfrag{nonmax}{\huge\text{(Non-maximal) Tassel}}
	\psfrag{2}{\Large$2$}
	\psfrag{3}{\Large$3$}
	\psfrag{4}{\Large$4$}
	\psfrag{5}{\Large$5$}
	\psfrag{6}{\Large$6$}
	\psfrag{2b-a1}{\Large$2b-a$}
	\psfrag{3b-2a}{\huge$3b-2a$}
	\psfrag{3b-2a1}{\Large$3b-2a$}
	\psfrag{4b-3a}{\large$4b-3a$}
	\centerline{\scalebox{.4}{\includegraphics{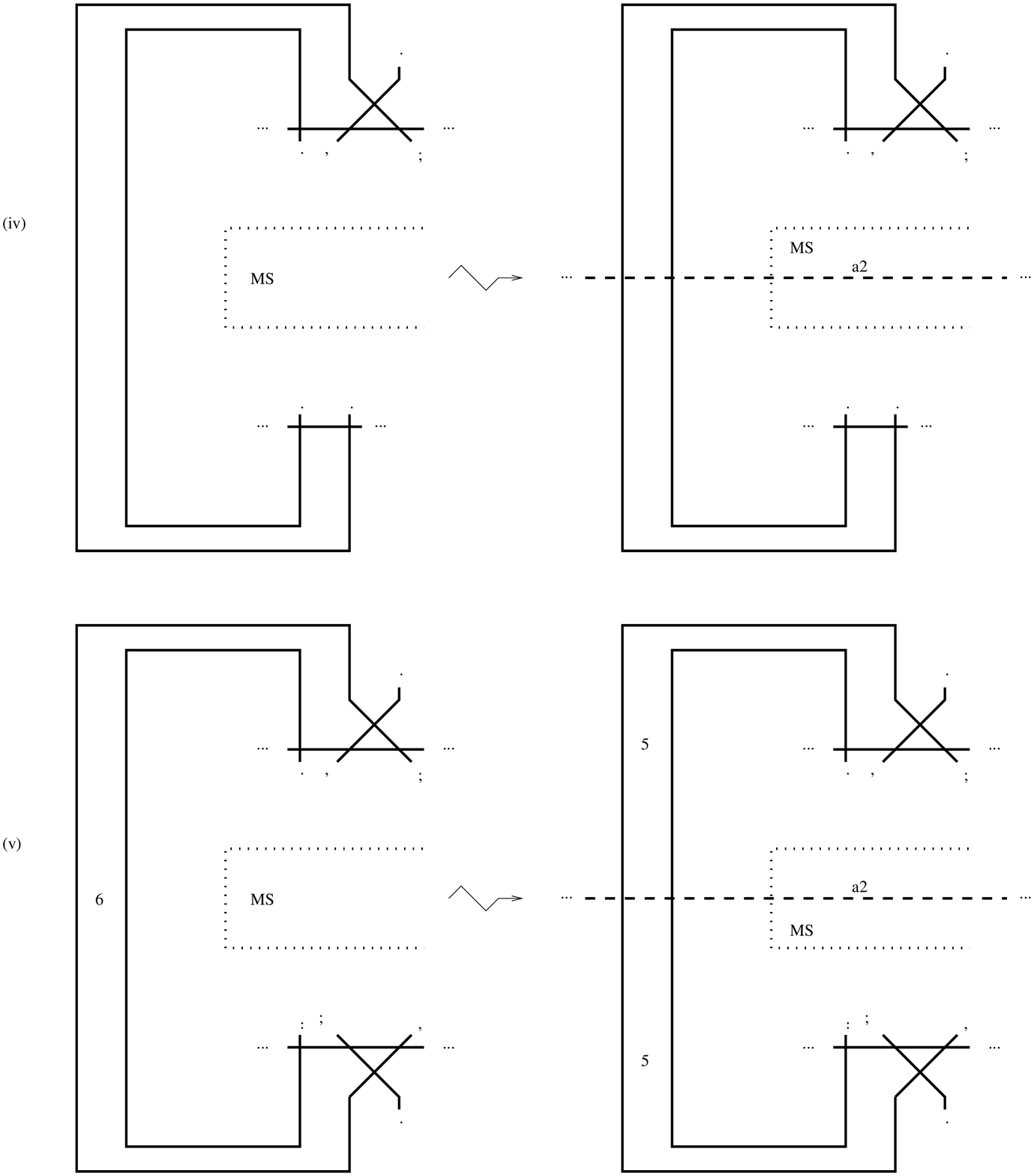}}}
	\caption{Possible instances for the top-to-bottom regions (concl'd).}\label{fig:toptobottomconc-new}
\end{figure}

\subsection{Consequences of the kinkification}\label{subsect:conseqkinks}

We observe that at the end of the Step $1$ of our kinkification of the braid closure in order to obtain a delta diagram, we obtain a braid closure again. It is only after Step $2$ of the process that we end up with a diagram other than a braid closure. It would be interesting if it would be possible to modify Step $2$ so that the final product would be a braid closure.

Furthermore, we have control over the number of crossings of our final product.

\begin{thm}
Let $L$ be a link and $n$ be the crossing number of $L$. Then the increase in the number of crossings of the delta diagram obtained by kinkification as described in the proof of Theorem \ref{thm:deltadiagram} $($when compared with the number of crossings of a minimal diagram of $L)$ is bounded above by
\[
N^2  + 4N + 3
\]
where $N = n + (n-1)(n-2)$
\end{thm}
\begin{proof}
We begin by noting that the number of Seifert circles for any diagram of $L$ is bounded above by the number of crossings. This follows from the fact that, since $L$ does not have trivial components, then the number of Seifert circles is maximized on a diagram where as we go along it, from each crossing there stems a Seifert circle.

Also, the braid index is bounded above by the number of crossings.

We now consider a  diagram of $L$ with $c$ crossings.   Then we apply Vogel's procedure (\cite{PVogel}) to obtain a closed braid from it along with the following bound for the number of crossings on the resulting braid closure:
\[
c + (s-1)(s-2) \leq c + (c-1)(c-2)
\]
where $s$ stands for the number of Seifert circles \cite{PVogel}.

Assuming our braid closure now has $N$ crossings we will now apply our kinkification procedure in order to obtain a delta diagram from it. In Step $1$, for each pair of crossings we introduce a kink ($\leq N/2$ more crossings) and each of these hover about the braid entanglement twice ($2\times  b \leq 2 N$, where $b$ stands for the number of strands in the braid). Then at the end of Step $1$ we have $\leq (N/2)\times 2N = N^2$ more crossings. Now for the contribution of Step $2$ -- the reader may want to view Figure \ref{fig:10-35-5KINK} again. Arc $2$ goes over the strands of the original braid (left, $b\leq N$ crossings produced). Then it goes over the arcs produced by the introduction of the kinks in Step $1$ (twice, actually, so $\leq 2\times N$ crossings produced). Then it goes over the middle section ($b \leq N$ crossings produced). Then over the arcs produced by the kinks introduced in Step $1$ again (already accounted for). Finally, it adds $3$ more crossings at the bottom of the diagram. The total contribution of Step $2$ for the increase on the number of crossings is then
\[
4N + 3
\]
Hence the kinkification increases the number of crossings, $N$, by adding not more than
\[
N^2 + 4N + 3
\]
crossings to the braid closure it was applied to.
We now compose this result with the preceding estimate of not more than
\[
c + (c-1)(c-2)
\]
crossings from the passage from diagram to closure of braid.
The proof is complete.
\end{proof}

\section{Discharging Procedures}\label{sect:discharging}

\begin{thm}\label{thm:dis}
In a connected lune-free diagram, 
there is at least one triangle which is adjacent  either to another triangle, or to a tetragon, or to a pentagon.
\end{thm}
\begin{proof}

Let $f_i$ stand for the number of regions with $i$ sides on a given diagram. Then, from Euler's formula (see \cite{EliahouHararyKauffman}),
\[
8 = \sum_{i\geq 2}(4-i)f_i
\]
which amounts to

\[
2f_2 + f_3 + \sum_{i\geq 5}(4-i)f_i= 8
\]
In particular, any lune-free diagram has $f_2=0$. Thus it amounts to
\[
f_3 + \sum_{i\geq 5}(4-i)f_i= 8
\]
Therefore any lune-free diagram has to have triangles since the remaining coefficients are negative or zero. We regard each of the coefficients in this formula as a {\it charge} associated with a region. Each region with $i$ sides is assigned the charge $4-i$. Note that only triangles have positive charge.

We are now in a position to make an argument by {\it discharging} (see \cite{appelhaken}). By discharging we mean taking a positive amount of charge from a region and placing it in an adjacent region. For example suppose we have a triangle with charge $1$ and an adjacent region with charge $c$. Change the charge on the triangle by subtracting $1/3$ and the charge in the adjacent region to $c+1/3$. The total sum of the charges remains the same.

Now suppose that no triangle in the diagram is adjacent to a $3$-gon, a $4$-gon or a $5$-gon. Consider  discharging triangles into regions with $i$ sides, $i\geq 6$, as described above. The worst case is when every edge of the region has a discharge. Then the new charge in the region would be $4-i+i/3 = \frac{12-2i}{3} \leq 0$ for $i\geq 6$. Thus by discharging all the triangles we would have the sum of all the charges  less than or equal to zero. This is a contradiction. And therefore we conclude that there must be at least one triangle which is adjacent to a $4$-gon or a $5$-gon.
\end{proof}

In Figure \ref{fig:6exceptional} we present an example of a planar tetravalent graph and of a link diagram without 2-sided regions and whose triangles do not share edges with tetragons nor pentagons.

\begin{figure}[!ht]
	\psfrag{3}{\huge$3$}
	\psfrag{6}{\huge$6$}
	\centerline{\scalebox{.4}{\includegraphics{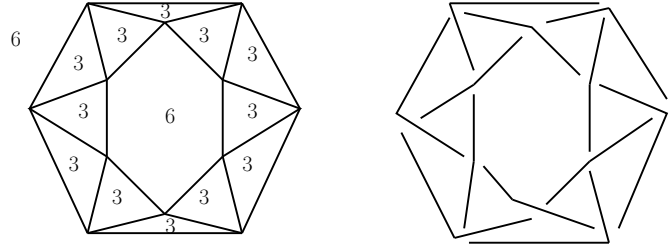}}}
	\caption{An example of a planar tetravalent graph and  of a link diagram without 2-sided regions and whose triangles do not share edges with tetragons nor pentagons.}\label{fig:6exceptional}
\end{figure}

\subsection{Feature invariants from the Discharging Procedure Theorem \ref{thm:dis}}\label{subsect:featureinvs}

In this section and the next we give a number of feature invariants of knots and links. A feature invariant is obtained from a combinatorial property of a diagram by minimizing over the appearance of the number of such properties in all diagrams that represent a given knot or link. For example, every diagram has a certain number of crossings. By minimizing the crossing number we obtain an invariant that is usually called the minimal crossing number of a knot of link. Here we have features such as adjacency of triangles with other triangles and we can minimize or maximize the count of such occurrences. There are many feature invariants suggested by the present work, and we hope that some of them will turn out to be significant for the theory of knots and links.

\begin{def.}
Number of triangles in a lune-free diagram, $D_L$, which share an edge with a triangle, $\Delta_3^{D_L}$, which share an edge with a tetragon, $\Delta_4^{D_L}$, or a pentagon $\Delta_5^{D_L}$,  or either with a triangle or a tetragon, $\Delta_{3, 4}^{D_L}$,  or either with a triangle or a pentagon, $\Delta_{3, 5}^{D_L}$, or either with a tetragon or a pentagon, $\Delta_{4, 5}^{D_L}$, or either with a triangle, or tetragon or a pentagon, $\Delta_{3, 4, 5}^{D_L}$. Their minima over all lune-free diagrams of $L$: $\Delta_3^{L}$, $\Delta_4^{L}$, $\Delta_5^{L}$, $\Delta_{3, 4}^{L}$, $\Delta_{3, 5}^{L}$, $\Delta_{4, 5}^{L}$, and $\Delta_{3, 4, 5}^{L}$.
\end{def.}

\begin{def.}
Number of triangles in the largest cluster of triangles, $C\Delta^{D_L}$ in a given diagram, $D_L$. Its minimum over all lune-free diagrams of $L$: $C\Delta^{L}$.
\end{def.}

\begin{def.}
Number of edges shared by a triangle and a non-triangle, $E\Delta^{D_L}$, over all triangles of diagram $D_L$. Its minimum over all lune-free diagrams of $L$: $E\Delta^{L}$.
\end{def.}

\begin{figure}[!ht]
	\psfrag{3}{\huge$3$}
	\psfrag{4}{\huge$4$}
	\centerline{\scalebox{.37}{\includegraphics{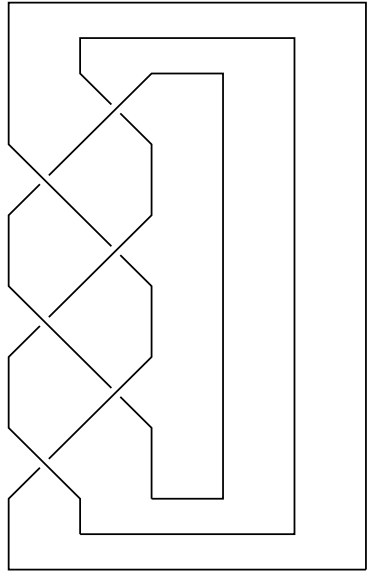}}}
	\caption{The Borromean rings or the Turk's head link THK(3, 3). Each of the regions of this diagram is a triangle. There is thus no adjacency of triangles to squares or pentagons.}\label{fig:trianglethknot}
\end{figure}

\begin{figure}[!ht]
	\psfrag{8-18}{\huge$8_{18}$}
	\psfrag{9-40}{\huge$9_{40}$}
	\psfrag{10-123}{\huge$10_{123}$}
	\psfrag{4}{\huge$4$}
	\psfrag{3}{\huge$3$}
	\psfrag{5}{\huge$5$}
	\centerline{\scalebox{.37}{\includegraphics{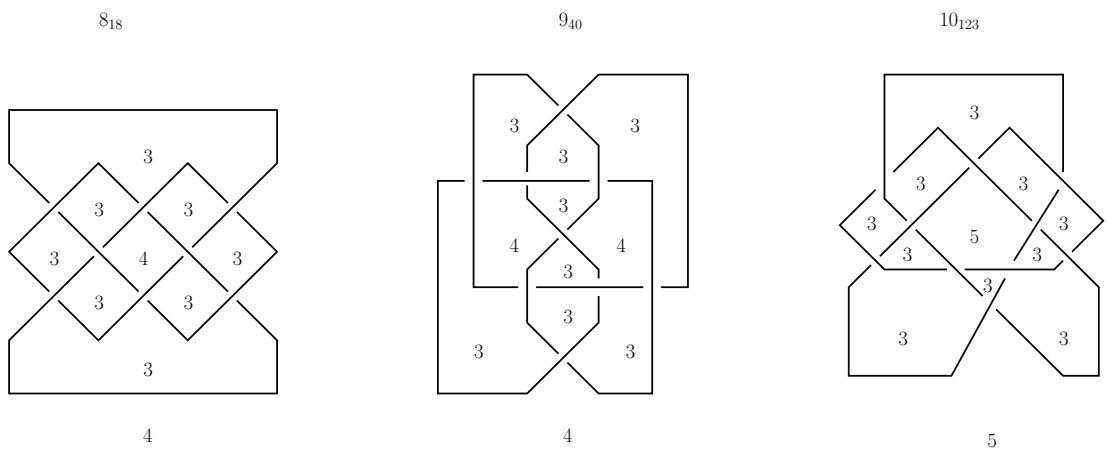}}}
	\caption{The prime knots $8_{18}$, $9_{40}$, and $10_{123}$ whose minimal diagrams, adapted from \cite{Rolfsen}, are delta diagrams. Numbers in regions count the corresponding sides.}\label{fig:8-18-9-40-10-123}
\end{figure}

\begin{figure}[!ht]
    \psfrag{11a-266}{\huge$11a_{266}$}
    \psfrag{12a-0868}{\huge$12a_{868}$}
	\psfrag{8-18}{\huge$8_{18}$}
	\psfrag{9-40}{\huge$9_{40}$}
	\psfrag{10-123}{\huge$10_{123}$}
	\psfrag{4}{\huge$4$}
	\psfrag{3}{\huge$3$}
	\psfrag{5}{\huge$5$}
	\centerline{\scalebox{.37}{\includegraphics{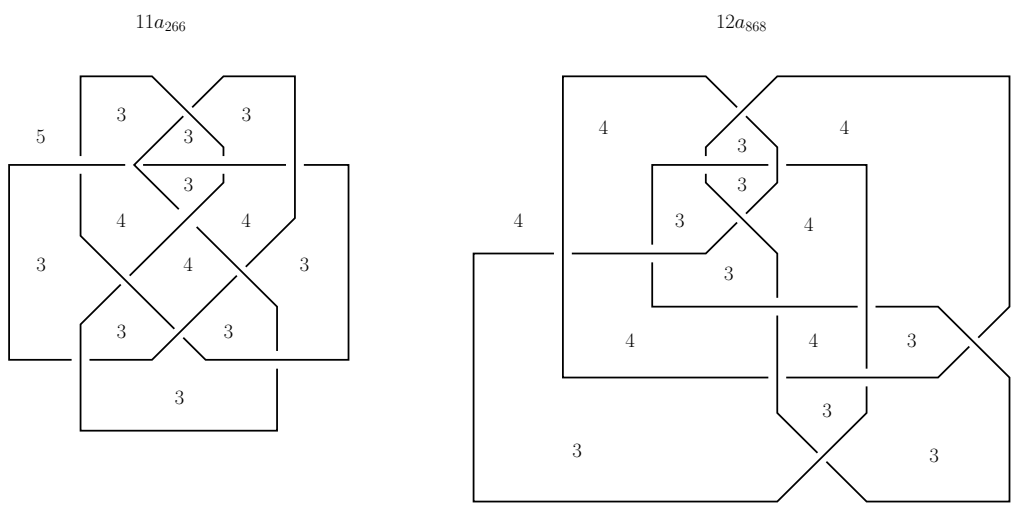}}}
	\caption{The prime knots $11a_{266}$ and $12a_{868}$ whose minimal diagrams, adapted from \cite{knotinfo}, are delta diagrams. Numbers in regions count the corresponding sides.}\label{fig:11a26612a868}
\end{figure}

\begin{figure}[!ht]
    \psfrag{12a-1019}{\huge$12a_{1019}$}
    \psfrag{12a-1188}{\huge$12a_{1188}$}
	\psfrag{8-18}{\huge$8_{18}$}
	\psfrag{9-40}{\huge$9_{40}$}
	\psfrag{10-123}{\huge$10_{123}$}
	\psfrag{4}{\huge$4$}
	\psfrag{3}{\huge$3$}
	\psfrag{5}{\huge$5$}
	\centerline{\scalebox{.37}{\includegraphics{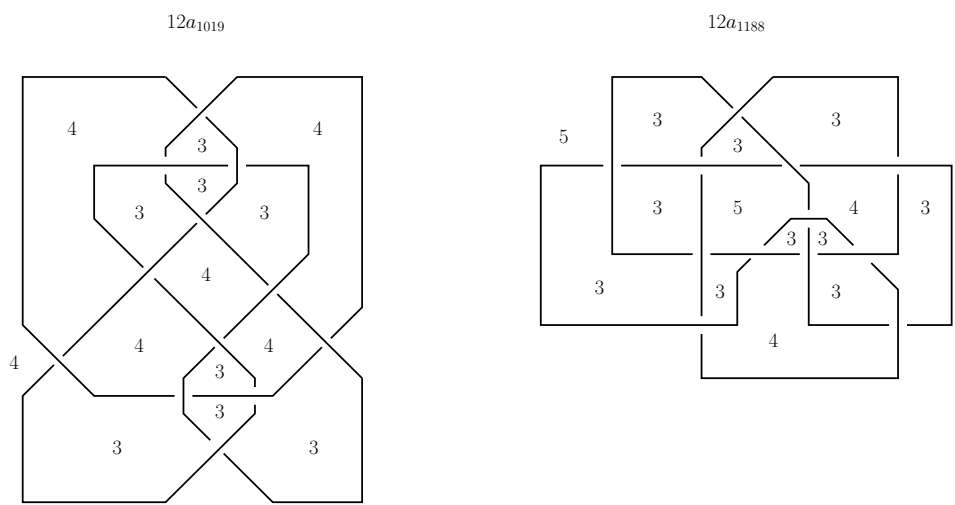}}}
	\caption{The prime knots $12a_{1019}$ and $12a_{1188}$ (prime determinant: $353$) whose minimal diagrams, adapted from \cite{knotinfo}, are delta diagrams. Numbers in regions count the corresponding sides.}\label{fig:12a101912a1188}
\end{figure}

\begin{figure}[!ht]
    \psfrag{12n-837}{\huge$12n_{837}$}
    \psfrag{12n-839}{\huge$12n_{839}$}
	\psfrag{8-18}{\huge$8_{18}$}
	\psfrag{9-40}{\huge$9_{40}$}
	\psfrag{10-123}{\huge$10_{123}$}
	\psfrag{4}{\huge$4$}
	\psfrag{3}{\huge$3$}
	\psfrag{5}{\huge$5$}
	\centerline{\scalebox{.37}{\includegraphics{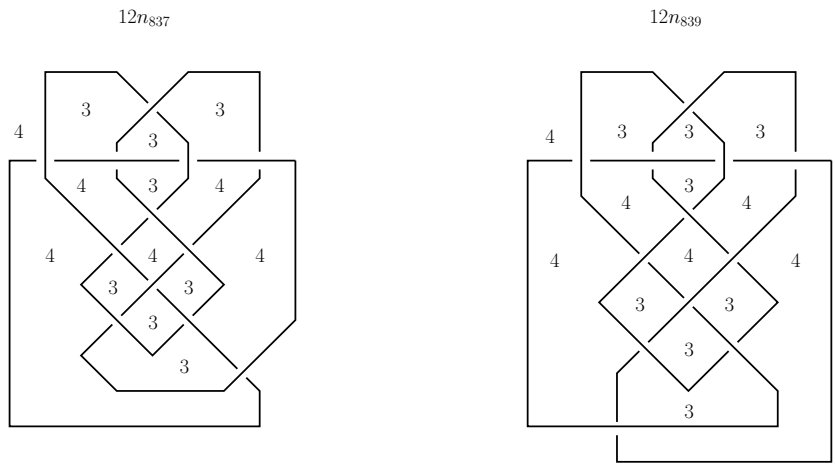}}}
	\caption{The prime knots $12n_{837}$ and $12n_{839}$ whose minimal diagrams, adapted from \cite{knotinfo}, are delta diagrams. Numbers in regions count the corresponding sides.}\label{fig:12n83712n839}
\end{figure}

\section{Feature Invariants}\label{sect:invs}

\begin{def.}
Let $L$ be a link. Let $f_k^{D_L}$ stand for the number of regions with $k$ sides in diagram $D_L$. We let $F_k(L)$ stand for the least upper bound of the ratio between $f_k^{D_L}$ over the total number of regions in diagram $D_L$ with the least upper bound taken over all diagrams of $L$.
\end{def.}

\begin{def.}\label{def:i-adj}
$3$-adjacency of diagram, $adj_3^{D_L}$, is the number of edges shared by $3$-sided regions.
We denote $Radj_3^{D_L}$ the quotient of $adj_3^{D_L}$ over the total number of edges in the diagram. Their minima over all diagrams of $L$ is denoted $adj_3^{L}$ and $Radj_3^{L}$. Analogously for $i$-sided regions.
\end{def.}

\begin{def.}Adjacency matrix of a diagram in the $i, j$ entry, denoted $adj_{(i, j)}^{D_L}$ is the number of edges that are adjacent to both an $i$-region and a $j$-region for that diagram. Its quotient over all edges in the diagram is denoted $Radj_{(i, j)}^{D_L}$. Their minima over all diagrams of $L$ is denoted $adj_{(i, j)}^{L}$ and $Radj_{(i, j)}^{L}$. When $i=j$ using the notation from Definition \ref{def:i-adj}.
\end{def.}

\begin{table}[h!]
\begin{center}
\renewcommand{\arraystretch}{1.25}
\begin{tabular}{| c | c | c | c | c | c | c |  c |  c | c | c  | }\hline
Knot \,&  $THK(3, 3)$ & $8_{18}$ & $9_{40}$ & $10_{123}$ & $11a_{266}$ & $12a_{868}$ & $12a_{1019}$ & $12a_{1188}$ &  $12n_{837}$ & $12n_{839}$     \\ \hline\hline
$\Delta_3^D$ &         $8$  & $8$ & $8$ & $10$ & $9$  & $8$ &  $8$&    $10$  & $8$ & $8$ \\ \hline
$\Delta_4^D$ &         $0$  & $8$ & $6$ & $0$ &  $7$  & $8$ &  $6$&    $6$  &  $6$ & $6$ \\ \hline
$\Delta_5^D$ &         $0$  & $0$ & $0$ & $10$ & $5$ &  $0$  & $0$    &$8$  &  $0$  &$0$ \\ \hline
$\Delta_{3, 4}^D$ &    $8$  & $8$ & $8$ & $10$ & $9$  & $8$ &  $8$&    $10$  & $8$ & $8$ \\ \hline
$\Delta_{3, 5}^D$ &    $8$  & $8$ & $8$ & $10$ & $9$  & $8$ &  $8$&    $10$  & $8$ & $8$ \\ \hline
$\Delta_{4, 5}^D$ &    $0$  & $8$ & $6$ & $10$ & $8$  & $8$ &  $6$&    $10$&   $6$ & $6$ \\ \hline
$\Delta_{3, 4, 5}^D$ & $8$  & $8$ & $8$ & $10$ & $9$  & $8$ &  $8$&    $10$  & $8$ & $8$ \\ \hline
$C\Delta^D$ &          $8$  & $8$ & $4$ & $10$ & $5$  & $4$ &  $4$&    $6$  &  $4$ & $4$ \\ \hline
$E\Delta^D$ &          $0$  & $8$ & $12$ & $10$ & $11$ &$10$&  $12$  & $14$ &  $12$& $12$ \\ \hline
$f_3^D$ &              $8$  & $8$ & $8$ & $10$ & $9$  & $8$ &  $8$&    $10$ &  $8$ & $8$ \\ \hline
$f_4^D$ &              $0$  & $2$ & $3$ & $0$ & $3$   & $6$ &  $6$&    $2$  &  $6$ & $6$ \\ \hline
$f_5^D$ &              $0$  & $0$ & $0$ & $2$ & $1$   & $0$ &  $0$&    $2$  &  $0$ & $0$ \\ \hline

$adj_3^D$ &      $12$ & $8$ & $6$ & $10$& $7$  & $6$ &  $6$&    $8$&  $6$ & $6$ \\ \hline
$adj_4^D$ &      $0$  & $0$ & $0$ & $0$ & $2$  & $6$ &  $6$&    $0$&  $6$ & $6$ \\ \hline
$adj_5^D$ &      $0$  & $0$ & $0$ & $0$ & $0$  & $0$ &  $0$&    $0$&  $0$ & $0$ \\ \hline
$adj_{4, 5}^D$ & $0$  & $0$ & $0$ & $0$ & $0$  & $0$ &  $0$&    $2$&  $0$ & $0$ \\ \hline
\end{tabular}
\caption{Several feature invariants for the minimal diagrams - which are delta diagrams - of the link and the knots in Figures \ref{fig:trianglethknot}, \ref{fig:8-18-9-40-10-123}, \ref{fig:11a26612a868}, \ref{fig:12a101912a1188}, and \ref{fig:12n83712n839}.}
\label{Ta:featureinvs}
\end{center}
\end{table}

\section{Questions}\label{sect:questions}

Many good questions arise in formulating lune-free and delta diagrams. Here are three questions that we would like to see answered.

\begin{enumerate}


\item What is the minimum number of colors of a link restricted to delta diagrams? Is it the same as over all diagrams of this link?

\item If the preceding is not true then is it true that knots admitting minimal colorings modulo one of $2$, $3$, $5$, or $7$, also admit delta diagrams supporting such a minimal coloring?

\item Determine for a given link K all minimal crossing delta diagrams.

\item Can we change the procedure in Step $2$ (or even already in Step $1$) of the kinkification so that the final product is the closure of a braid?

\item Find an upper bound for the number of faces produced by the kinkification in terms of crossing number.
\end{enumerate}

\end{document}